\documentclass[10pt,leqno]{amsart}
\usepackage{hyperref}
\usepackage[all,cmtip]{xy}
\usepackage{a4,latexsym}
\usepackage[utf8]{inputenc}
\usepackage{textcomp,fontenc}
\usepackage{exscale}
\usepackage{amsxtra,amsmath,amsfonts,amscd, amssymb, mathrsfs, amsthm}
\usepackage{color}
\usepackage{mathtools}

\newtheorem{theo}{Theorem}[section]

\newtheorem{lemma}[theo]{Lemma}

\theoremstyle{definition}

\newtheorem{rem}[theo]{Remark}

\theoremstyle{plain}

\newtheorem{conjecture}[theo]{Conjecture}

\theoremstyle{definition}
\newtheorem{defi}[theo]{Definition}

\theoremstyle{remark}

\numberwithin{paragraph}{section}
\numberwithin{equation}{section}

\usepackage{color}
\usepackage{comment}

\def\P{{\mathbb P}}
\def\N{{\mathbb N}}
\def\Z{{\mathbb Z}}
\def\Q{{\mathbb Q}}
\def\R{{\mathbb R}}

\def\KO{{\mathcal O}}

\def\KZ{{\mathcal Z}}

\def\SA{{\mathscr A}}

\def\SC{{\mathscr C}}
\def\SD{{\mathscr D}}

\def\SF{{\mathscr F}}

\def\SH{{\mathscr H}}

\def\SJ{{\mathscr J}}

\def\SL{{\mathscr L}}
\def\SM{{\mathscr M}}
\def\SN{{\mathscr N}}
\def\SO{{\mathscr O}}

\def\SU{{\mathscr U}}
\def\SX{{\mathscr X}}
\def\SY{{\mathscr Y}}

\def\Fm{{\mathfrak m}}

\def\Fa{{\mathfrak a}}
\def\Fb{{\mathfrak b}}

\def\BFB{{\mathbf B}}
\def\kcirc{{K^\circ}}

\def\an{{\rm an}}

\def\Spec{{\rm Spec}}
\def\Xan{{X^\mathrm{an}}}
\def\MA{{\rm MA}}

\newcommand{\metr}{{\|\hspace{1ex}\|}}

\title[Monge--Amp\`ere in mixed characteristic]{On the non-Archimedean Monge--Amp\`ere equation in mixed characteristic}

\author[Y.~Fang]{Yanbo Fang}
\address{Y. Fang, Mathematik, Universit{\"a}t 
Regensburg, 93040 Regensburg, Germany}
\email{yanbo.fang@mathematik.uni-regensburg.de}

\author[W.~Gubler]{Walter Gubler}
\address{W. Gubler, Mathematik, Universit{\"a}t 
Regensburg, 93040 Regensburg, Germany}
\email{walter.gubler@mathematik.uni-regensburg.de}

\author[K.~K\"unnemann]{Klaus K{\"u}nnemann}
\address{K. K{\"u}nnemann, Mathematik, Universit{\"a}t 
Regensburg, 93040 Regensburg, Germany}
\email{klaus.kuennemann@mathematik.uni-regensburg.de}

\date{\today}

\thanks{Y. Fang, W.~Gubler and K.~K{\"u}nnemann 
were supported by the collaborative research 
center SFB 1085 \emph{Higher Invariants - Interactions between Arithmetic Geometry and Global Analysis} funded by the Deutsche Forschungsgemeinschaft.}

\begin{document}

\begin{abstract}
Let $X$ be a smooth projective variety over a complete discretely valued field of mixed characteristic.
We solve non-Archimedean Monge--Ampère equations on $X$ assuming resolution and embedded resolution of singularities.
We follow the variational approach of Boucksom, Favre, and Jonsson proving the continuity of the plurisubharmonic envelope of a continuous metric on an ample line bundle on $X$.
We replace the use of multiplier ideals in equicharacteristic zero by the use of 
perturbation friendly test ideals introduced by Bhatt, Ma, Patakfalvi, Schwede, Tucker, Waldron, and Witaszek 
building upon previous constructions by Hacon, Lamarche, and Schwede.

\bigskip

\end{abstract}

\maketitle
\setcounter{tocdepth}{1}

\tableofcontents

\section{Introduction}

Let $L$ be an ample line bundle on an $n$-dimensional projective variety $X$ over a non-Archimedean field $K$. The non-Archimedean analogue of the famous \emph{Calabi--Yau problem} asks for a given  Radon measure $\mu$ on the Berkovich analytification $\Xan$ with $\mu(\Xan)=\deg_L(X)$ whether there exists a continuous semipositive metric $\metr$ of $L$, unique up to scaling, such that 
\begin{equation} \label{MA equation}
c_1(L,\metr)^{\wedge n}=\mu
\end{equation}
using the non-Archimedean Monge--Amp\`ere measure introduced by Chambert-Loir on the left. 
We call  \eqref{MA equation} the \emph{non-Archimedean Monge--Amp\`ere equation}. The analogous problem over the complex numbers was solved by Yau for Radon measures given by smooth volume forms within the class of smooth metrics. Uniqueness was shown before by Calabi. Later Kołodziej used pluripotential theory to treat more singular measures and solutions \cite{kol98:_monge_amper}.

Yuan and Zhang proved uniqueness up to scaling for solutions of the non-archime\-dean Monge--Amp\`ere equation \cite{yuan-zhang}.  
In a groundbreaking work, Bouck\-som, Favre and Jonsson solved the non-Archimedean Monge--Amp\`ere equation over a complete discretely valued field $K$ of \emph{residue characteristic zero} assuming that $X$ is smooth and that the support of the Radon measure $\mu$ is contained in the skeleton of an SNC-model (a projective regular model with special fiber having simple normal crossing support) \cite{bfj-solution}. 
They assumed also that $K$ is a completion of the function field of a curve at a closed point. 
This geometric condition was later removed by Burgos, Jell, Martin and the last two authors of this paper \cite{BGJKM}. 
Boucksom, Favre, and Jonsson used a variational approach 
to solve \eqref{MA equation} which relies crucially on the continuity of the semipositive envelope of a continuous metric of $L$. 
They show continuity of this envelope by using multiplier ideals on SNC-models \cite{bfj-singular}. 
In their arguments, Hironaka's resolution of singularities plays an important role in order to have sufficiently many SNC-models at hand. 
Apart from the multiplier ideals,  it is precisely here where  residue characteristic zero is used again.

In equicharacteristic $p>0$, the following existence result was shown by Jell, Martin and the last two authors of this paper \cite{gubler-jell-kuennemann-martin}. 
Similarly as above, it is assumed that $K$ is a completion of the function field of a curve over a perfect field $k$ of characteristic $p>0$. 
It is also assumed that resolution of singularities and embedded resolution of singularities hold in dimension $n+1$, see \S \ref{resolution of singularities} for precise definitions.  
Then existence of a solution of \eqref{MA equation} is shown in \cite{gubler-jell-kuennemann-martin} if  the support of $\mu$ is contained in the skeleton of an SNC-model. 
The proof is along the same lines as in \cite{bfj-singular,bfj-solution} replacing multiplier ideals by test ideals. 
Note that this result is unconditional for $n=2$ by using resolution of singularities for three-folds proved by Cossart and Piltant.

If $K$ has mixed characteristic, then there are results about the non-archime\-de\-an Monge--Ampère problem for varieties like curves, abelian varieties and toric varieties based on their special geometry \cite{thuillier-thesis,liu2011, burgos-gubler-jell-kuennemann2021}.

In this paper, we deal with the Monge--Ampère problem for arbitrary smooth projective varieties over a complete discretely valued field $K$ of mixed characteristic $(0,p)$. 
For a complete local noetherian domain $R$ of mixed characteristic, a theory of test ideals was introduced by Ma and Schwede based on perfectoid ideas \cite{ma-schwede2018a, ma-schwede2018b}.
Using perfectoid methods, prismatic techniques, and a $p$-adic Riemann-Hilbert correspondence, Bhatt was able in 2020 to show a variant of Kodaira vanishing `up to finite covers' in mixed characteristic \cite{bhatt2020}. 
Applications to the minimal model program were given by Bhatt, Ma, Patakfalvi, Schwede, Tucker, Waldron, and Witaszek \cite{bmpstww2021} and independently by Takamatsu and Yoshikawa \cite{ty2021}. 
For projective normal schemes over $R$, these ideas were extended by Hacon, Lamarche and Schwede. 
They introduced \emph{$+$-test ideals} and showed global generation results \cite{hls2021}. 
Their construction, which replaces spaces of global sections by so-called spaces of $+$-stable sections from\cite{bmpstww2021,ty2021}, is recalled in \S \ref{plus-test-ideals}.
Hacon, Lamarche and Schwede conjecture that subadditivity holds for their $+$-test ideals \cite[Conjecture 8.3]{hls2021}. 
Observe that subaddivity is well known for multiplier resp.~test ideals in the equicharacteristic case. 
While this conjecture remains open, a modified version of $+$-test ideals, which we call \emph{perturbation friendly global test ideals} in this article, was introduced by Bhatt, Ma, Patakfalvi, Schwede, Tucker, Waldron, and Witaszek  \cite{schwede-etal-2024}, again benefiting from a $p$-adic Riemann-Hilbert correspondence. 
These perturbation friendly test ideals enjoy properties as nice as the $+$-test ideals of Hacon, Lamarche and Schwede, and in addition satisfy the subadditivity property. 
We refer to \S \ref{perturbation-friendly-global-test-ideals} for details.

The contribution of this paper is to show that global $+$-test ideals and perturbation friendly global test ideals allow applications to the non-Archimedean Monge--Amp\`ere problem in mixed characteristic. 
More precisely, we prove the following results.

\newpage

\begin{theo} \label{intro: continuity of the envelope}
Let $L$ be an ample line bundle on a smooth projective variety $X$ over $K$. 
\begin{enumerate}
\item
If $\metr$ is a model metric on $L$ induced by a model $(\SX,\SL)$ of $(X,L)$ with $\SX$ regular
and if $L$ has an ample model $\mathscr A$ on $\SX$, then the semipositive envelope of the metric $\metr$ is continuous. 
\item
If resolution of singularities holds for  projective $\kcirc$-models of $X$, then the semipositive envelope of any continuous metric $\metr$ of $L$ is continuous. 
\end{enumerate}
\end{theo}

This follows from Theorem \ref{relative Theorem 8.5} and Theorem \ref{continuity-of-the-envelope} by Remark \ref{dictionary}.
This result is the key to apply the variational method of Boucksom--Favre--Jonsson as we will see in Theorem \ref{Calabi-Yau theorem}.

\begin{theo} \label{into Calabi-Yau theorem}
Let $L$ be an ample line bundle on an $n$-dimensional smooth projective variety $X$ over $K$. 
If resolution of singularities holds for projective models of $X$ and embedded resolution of singularities holds for regular projective models of $X$, then the non-Archimedean Monge--Amp\`ere equation \eqref{MA equation} is solved by a continuous semipositive metric $\metr$ of $L$, unique up to scaling, if the positive Radon measure $\mu$ has support in the skeleton of an SNC-model of $X$.
\end{theo}

\section{Model metrics, curvature forms and psh envelopes}

Non-Archimedean pluripotential theory in higher dimensions was introduced by Boucksom, Favre and Jonsson \cite{bfj-singular}.
We recall basic notions of non-Archimedean pluripotential theory following \cite[2.1-2.8]{gubler-jell-kuennemann-martin}.

Let $X$ be a proper variety over a non-Archimedean field $K$. 
Recall that the topological space underlying the \emph{Berkovich analytification $X^{\mathrm{an}}$ of $X$} consists of pairs $(p,|\phantom{a}|_p)$ with $p\in X$ and $|\phantom{a}|_p$ an absolute value on $\kappa(p)=\mathcal O_{X,p}/\mathfrak m_{X,p}$ that extends $|\phantom{a}|_K$ and is equipped with coarsest topology such that the map
$\pi\colon X^{\rm an}\rightarrow X$, $(p,|\phantom{a}|_p) \mapsto p$
is continuous and for all Zariski open subsets $U$ of $X$ and all regular functions $f \in \mathcal O_X(U)$ the map
$|f|\colon U^{\mathrm{an}}=\pi^{-1}(U)\rightarrow \R,\,(p,| \phantom{a}|_p)  \mapsto |f(p)| := |f+ 
\mathfrak m_{X,p}|_p$ is continuous as well.

A \emph{model} of $X$ is a proper flat scheme $\SX$ over $S\coloneqq \mathrm{Spec}~K^{\circ}$ together with an isomorphism $h$ from $X$ to the generic fiber $\SX_{\eta}$ of the $S$-scheme $\SX$. 
The special fiber of $\SX$ is denoted by $\SX_s$.
Given a model $\SX$, the valuative criterion for properness yields a natural reduction map $\mathrm{red}_\SX\colon X^\an\to \SX_s$.
Let $L$ be a line bundle over $X$.
A \emph{model} of $(X,L)$ is given by a model $\SX$ of $X$ and a line bundle $\SL$ over $\SX$ with an isomorphism of $L$ to $h^*(\SL|_{\SX_{\eta}})$.
A \emph{continuous metric} $\| \ \|$ of $L$ associates with each section $s\in \Gamma(U,L)$ on some Zariski
open subset $U$ of the variety $X$ a continuous function $\|s\|\colon  U^{\mathrm{an}}\to [0,\infty)$ such that one has $\|f\cdot s\|=|f|\cdot \|s\|$ for each regular function $f$ in $\KO_X(U)$. It is furthermore required that $\|s\|>0$, if $s$ is a nowhere vanishing section of $L$.

Let $(\SX,\SL)$ be a model of $(X,L^{\otimes m})$ for some $m\in\N_{>0}$.
The \emph{model metric} $\| \ \|_{\SL}$ of $L$ over $X^{\an}$ is determined by $\|s\|_{\SL}:= \sqrt[m]{|g|}$ on $U^{\an}\cap\mathrm{red}^{-1}(\SU_s)$, where $\SU$ is open in $\SX$, $s$ is a section of $L$ over $U\coloneqq X\cap\SU$, $b$ is a nowhere vanishing section of $\SL$ over $\SU$, $g\in \KO_X(U)$ is a regular function such that $s^{\otimes m}=gb$ over $U$, and $\mathrm{red}_\SX\colon X^\an\to \SX_s$ is the reduction map. 
A model metric $\| \ \|$ on $\KO_X^{\an}$ induces a so-called \emph{model function} $-\log\|1\|\colon X^{\an}\to \R$. 
The $\Q$-vector space of all model functions on $X$ is denoted by $\SD(X)$. 

Let $\mathfrak a$ be an ideal sheaf on a model $\SX$ of $X$ which is supported in the special fibre.
The exceptional divisor $E$ of the blowup $\SY$ of $\SX$ in $\mathfrak a$ defines a model $(\SY,\KO_{\SY}(E))$ of $(X,\KO_X)$.
The associated model function in $\SD(X)$ is denoted by $\log |\mathfrak a|$.

Given a model $\SX$ of $X$ one defines $N^1(\SX/S)_{\Q}$ as the quotient of the $\Q$-vector space $\mathrm{Pic}(\SX)_{\Q}:=\mathrm{Pic}(\SX)\otimes_\Z\Q$ by the subspace generated by line bundles whose restriction to every closed curve $C$ in the special fiber $\SX_s$ has degree zero.
Call $\alpha\in N^1(\SX/S)\coloneqq N^1(\SX/S)_\Q\otimes_{\Q}\R$ 
\emph{nef} if $\alpha\cdot C\geq 0$ for all such curves.

The \emph{space of closed (1,1)-forms on $X$} is defined as
\[
\KZ^{1,1}(X)\coloneqq \R\otimes_\Q\varinjlim N^1(\SX/S)_{\Q}
\]
where the direct limit is taken over all isomorphism classes of models of $X$ and the transition maps are induced by pullback between dominating models.
For the model metric induced by a model $\SL$ of $L^{\otimes m}$, its \emph{curvature form} $c_1(L,\| \ \|_{\SL})$ is the image of $\SL^{\otimes \frac{1}{m}}\in N^1(\SX/S)_{\Q}$ in $\KZ^{1,1}(X)$. 
By construction we have a map $dd^c\colon \SD(X)\to \KZ^{1,1}(X)$ given by $g\mapsto c_1(\KO_X,\| \ \|\cdot \mathrm{e}^{-g})$. 
Call a closed $(1,1)$-form $\theta\in \KZ^{1,1}(X)$ \emph{semipositive} if it can be represented by a nef element in $N^1(\SX/S)$ for some model $\SX$. 
Denote by $N^1(X)$ the quotient of $\mathrm{Pic}(X)_\R:=\mathrm{Pic}(X)\otimes_\Z\R$ by numerical equivalence.
The  map $\{\phantom{a}\}\colon \KZ^{1,1}(X)\to N^1(X)$ induced by the restriction maps $N^1(\SX/S)\to N^1(X)$ sends a closed $(1,1)$-form $\theta$ to its  \emph{de Rham class} $\{\theta\}$.  
A class in $N^1(X)$ is called \emph{ample} if it is an $\R_{>0}$-linear combination of classes induced by ample line bundles on $X$.

We fix $\theta\in \KZ^{1,1}(X)$. 
A model function $\varphi\in \SD(X)$ is called \emph{$\theta$-plurisubharmonic} (\emph{$\theta$-psh} for short) if the class $\theta+dd^c\varphi\in \KZ^{1,1}(X)$ is semipositive.
The space of all psh model functions is denoted by $\mathrm{PSH}_{\SD}(X,\theta)$. For a continuous function $u\in \SC^0(X^{\an})$, its \emph{$\theta$-psh envelope} $P_{\theta}(u)\colon X^{\an}\to\R\cup\{-\infty\}$ is the function defined 
by
\[
P_{\theta}(u)(x)\coloneqq\sup\bigl\{\varphi(x)\ \big|\  \varphi\in \mathrm{PSH}_{\SD}(X,\theta),\ \varphi\leq u\bigr\},~(x\in X^\an) .
\]

For $u\in C^0(X^\an)$, $v\in \SD(X)$ and $t\in \R_{>0}$ we have
\begin{eqnarray}
\label{proposition-2.9-4}
P_\theta(u)-v&=&P_{\theta+dd^cv}(u-v),\\
\label{proposition-2.9-7}
P_{t\theta}(t\theta)&=&t P_\theta(u).
\end{eqnarray}
If the de Rham class $\{\theta\}\in N^1(X)$ is \emph{ample}, 
then $\mathrm{PSH}_\SD(X,\theta)$ is non-empty, $P_{\theta}(u)$ takes value in $\R$, and we have
\begin{equation}
\label{proposition-2.9-5}
\sup_{X^\an}|P_\theta(u)-P_\theta(u')|\leq \sup_{X^\an}|u-u'|.
\end{equation}
for all $u,u'\in C^0(X^\an)$.
For further properties of $\theta$-psh model functions and the $\theta$-psh envelope we refer to \cite{bfj-singular} and \cite[Section 2]{gubler-jell-kuennemann-martin}

For $\theta\in \KZ^{1,1}(X)$ that has an ample de Rham class $\{\theta\}$,
we consider $\varphi\in \mathrm{PSH}_\SD(X,\theta)$ and we assume that there exists a normal model $\SX$ of $X$ such that $\theta$ and $\varphi$ are induced by line bundles $\SM$ and $\SL$ on $\SX$.
The \emph{Monge--Ampère measure} $\mathrm{MA}_\theta(\varphi)$ is the discrete measure 
\[
\mathrm{MA}_\theta(\varphi)
\coloneqq 
\sum_{V}\mathrm{length}_{\KO_{\SX_s,\xi_V}}(\KO_{\SX_s,\xi_V})\,
\mathrm{deg}_{\SM\otimes\SL}(V)\,\delta_{x_V}
\]
on $X^{\rm an}$ where $V$ runs through the irreducible components of $\SX_s$, $x_V\in X^{\an}$ is the unique preimage of the generic point $\xi_V$ of $V$ under the reduction map $\mathrm{red}_\SX\colon X^\an\to \SX_s$, and $\delta_{x_V}$ denotes the Dirac probability measure supported in $x_V$.
For generalization of the Monge--Ampère measure to all $\theta\in \KZ^{1,1}(X)$ with ample de Rham class $\{\theta\}$ and to more general classes of $\theta$-psh functions, we refer to \cite{bfj-solution}, \cite{gubler-jell-kuennemann-martin} and Section \ref{section:MA-equation}.

\section{Global test ideals in mixed characteristic}\label{plus-test-ideals}
In this section, we introduce the global $+$-test ideals defined and studied by Hacon, Lamarche and Schwede in \cite{hls2021}.
The $+$-test ideals form a mixed characteristic analogue of the theories of multiplier ideals in characteristic zero and test ideals in positive characteristic. 
The theory is based on the notion of $+$-stable sections introduced in \cite{bmpstww2021,ty2021}. 

Let $(R,\Fm,k)$ be a complete noetherian local ring of mixed characteristic.
Let $p>0$ denote the characteristic of the residue field $k$.
Let $\SX$ be a normal integral scheme  which is proper over $S\coloneqq \mathrm{Spec}\, R$. 
The canonical sheaf $\omega_\SX$ is reflexive (\cite[\href{https://stacks.math.columbia.edu/tag/0AWK}{Tag 0AWK}]{stacks-project} and \cite[Theorem 1.9]{hartshorne1994}) and we fix a canonical divisor $K_\SX$ with $\omega_\SX=\KO_\SX(K_\SX)$ \cite[\href{https://stacks.math.columbia.edu/tag/0EBM}{Tag 0EBM}]{stacks-project}.

\begin{defi}\label{definition-of-B}
For a reflexive sheaf $\SM=\KO_\SX(M)$ associated with a divisor $M$ and an effective $\Q$-divisor $B$ on $\SX$,
the \emph{subspace of $+$-stable sections} of the adjoint line bundle $\omega_\SX\otimes \SM$ relative to $B$
\[
\mathbf{B}^0\bigl(\SX,B,\KO_\SX(K_\SX+M)\bigr)
\subset H^0\bigl(\SX,\KO_\SX(K_\SX+M)\bigr)
\]
is defined by
\[
\bigcap_{f\colon \SY\to \SX}\mathrm{Im}
\Bigl(H^0\bigl(\SY,\KO_\SY(K_\SY+\lceil f^*(M-B)\rceil\bigr)
\longrightarrow
H^0\bigl(\SX,\KO_\SX(K_\SX+M)\bigr)\Bigr),
\]
where an algebraic closure $\overline{\kappa(\SX)}$ of the 
function field $\kappa(\SX)$ is fixed, and $f$ runs through the finite surjective 
morphisms $f\colon \SY\to \SX$ from a normal integral scheme
$\SY$ together with an embedding 
$\kappa(\SY)\hookrightarrow \overline{\kappa(\SX)}$ \cite[Definition 3.2 and Lemma 3.8]{hls2021}.
\end{defi}

\begin{rem} \label{remark in Cartier case}
If $M-B$ is $\Q$-Cartier, then Definition \ref{definition-of-B}
holds as well if $f$ runs through the alterations $f\colon \SY\to \SX$ from a normal integral scheme $\SY$ together with an embedding $\kappa(\SY)\hookrightarrow \overline{\kappa(\SX)}$ \cite[Lemma 3.8]{hls2021}.
\end{rem}

For the rest of this section, we assume that the scheme $\SX$ is regular and projective over $S$.

\begin{defi}\label{defi-3.3}
Take a very ample line bundle $\SH$ on $\SX$. 
For an effective $\Q$-divisor $B$ on $\SX$ and for each $i\in \N$ the subspace
\begin{equation}\label{defi-plus-stable-sections}
\BFB^0\bigl(\SX,B,\omega_\SX\otimes \SH^i\bigr)
\subset H^0\bigl(\SX,\omega_\SX\otimes \SH^i\bigr)
\end{equation}
generates the subsheaf $\SN_i\subset \omega_\SX\otimes \SH^i$ which defines
\begin{equation}\label{def-test-ideal}
\SJ_i\coloneqq \SN_i\otimes \SH^{-i}\subset \omega_\SX.
\end{equation}
The sequence $(\SJ_i)_{i\in \N}$ is increasing and
becomes stationary.
Define the \emph{$+$-test submodule} $\tau_+(\omega_\SX,B)$ to be $\SJ_i$ for $i\gg 0$ \cite[Definition 4.3]{hls2021}.
\end{defi}

\begin{rem} \label{remarks on test ideals}
(i) Definition \ref{defi-3.3} does not depend on the choice of $\SH$ \cite[Proposition 4.5]{hls2021}.

(ii) For $i\gg 0$ we have
$H^0(\SX,\tau_+(\omega_\SX,B)\otimes \SH^i)=
\BFB^0(\SX,B,\omega_\SX\otimes \SH^i))$
\cite[Proposition 4.7]{hls2021}.
If $B'\geq B$,
then $\tau_+(\omega_\SX,B')\subset \tau_+(\omega_\SX,B)$
and if $F$ is an effective Cartier divisor, then \cite[Lemma 4.8]{hls2021}
\begin{equation}\label{plus-formula}
\tau_+(\omega_\SX,B+F)=\tau_+(\omega_\SX,B)\otimes \KO_\SX(-F).
\end{equation}

(iii) Equality \eqref{plus-formula} allows one to define $\tau_+(\omega_\SX,B)$ for a not necessarily effective $\Q$-divisor $B$ as 
$\tau_+(\omega_\SX,B+F)\otimes \KO_\SX(F)$ where $F$ is a 
Cartier divisor such that $B+F$ is effective \cite[Definition 4.14]{hls2021}.
\end{rem}

\begin{defi}
For an effective $\Q$-divisor $B$ on $\SX$, the ideal sheaf 
\begin{equation}\label{def-test-ideal2}
\tau_+(\KO_\SX,B)\coloneqq
\tau_+(\omega_\SX,K_\SX+B)
\end{equation}
is the called the {+-}\emph{test ideal associated with $B$} \cite[Definition 4.15, Lemma 4.18]{hls2021}. 
Using $F=-K_\SX$ in Remark \ref{remarks on test ideals}(iii), this is indeed a coherent ideal sheaf.
\end{defi}

We  recall the subadditivity conjecture of Hacon, Lamarche and Schwede \cite[Conjecture 8.3]{hls2021}.

\begin{conjecture}\label{hacon-et-al-8-3}
Given effective $\Q$-divisors $D$ and $E$ on $\SX$, we have
\[
\tau_+(\KO_\SX,D+E)\subset \tau_+(\KO_\SX,D)\cdot \tau_+(\KO_\SX,E).
\]
\end{conjecture}

The above conjecture is still open, but Bhatt, Ma, Patakfalvi, Schwede, Tucker, Waldron and Witaszek prove subadditivity for their new perturbation-friendly test ideals which we will introduce in the next section.

\section{Perturbation friendly global test ideals}\label{perturbation-friendly-global-test-ideals}

In this section, we present the perturbation friendly global test ideals after Bhatt, Ma, Patakfalvi, Schwede, Tucker, Waldron, and Witaszek \cite{schwede-etal-2024}. They are suitable for our purpose as subadditivity is known for them in contrast to the test ideals considered in the previous section.

Let $(R,\Fm,k)$ be a complete discrete valuation ring of mixed characteristic with field of fractions $K$.
We fix a flat projective integral regular scheme $\SX$ over $S=\Spec(R)$.
We fix a canonical divisor $K_\SX$ with $\omega_\SX=\KO_\SX(K_\SX)$.

\begin{defi} \label{def properties new test modules}
Let $B$ be a $\Q$-divisor on $\SX$. 
By \cite[Proposition 7.14(f)]{schwede-etal-2024}, there is a unique \emph{perturbation friendly test module} $\tau(\omega_\SX,B)$ such that there exists an effective Cartier divisor $G^0$ with the property that 
\begin{equation}\label{define-pert-test-module}
\tau(\omega_\SX,B)= \tau_+(\omega_\SX,B+ \varepsilon G)
\end{equation}
for all divisors $G \geq G^0$ on $\SX$ and all $0<\varepsilon \ll 1$ (depending on $G$).

We define the \emph{perturbation friendly test ideal}
\begin{equation}\label{define-pert-test-ideal}
\tau(\KO_\SX,B)\coloneqq \tau(\omega_\SX,K_\mathscr{X}+B).
\end{equation}
Both the test module $\tau(\omega_\SX,B)$ and the test ideal $\tau(\KO_\SX,B)$ are coherent \cite[Definition 7.12, Theorem 7.13]{schwede-etal-2024}. If $B$ is effective, then the fractional ideal sheaf $\tau(\KO_\SX,B)$ is indeed an ideal sheaf in $\KO_{\SX}$ \cite[Definition 7.18]{schwede-etal-2024}.
\end{defi}

Here are some properties of perturbation friendly test ideals $\tau(\KO_\SX,B)$ including subadditivity which holds only conjecturally \cite[Conjecture 8.3]{hls2021} for the test ideals $\tau_+(\KO_\SX,B)$.

\begin{theo}[Bhatt, Ma, Patakfalvi, Schwede, Tucker, Waldron, Witaszek]
\label{more properties new test modules}
Let $B, B_1,B_2$ be $\Q$-divisors on $\SX$.

(i) There exists a Cartier divisor $G^0$ on $\SX$ with the property that 
\[
\tau(\KO_\SX,B)= \tau_+(\KO_\SX,B+ \varepsilon G)\]
for all divisors $G \geq G^0$ on $\SX$ and all $0<\varepsilon \ll 1$ (depending on $G$).

(ii) If $B$ is a $\Q$-divisor and $F$ is a divisor on $\SX$, then 
\begin{align}
\label{noplus-formula}
\tau(\omega_\SX,B+F)=\tau(\omega_\SX,B)\otimes \KO_\SX(-F),\\
\label{noplus-formula2}
\tau(\KO_\SX,B+F)=\tau(\KO_\SX,B)\otimes \KO_\SX(-F).
\end{align}

(iii) \emph{(Subadditivity)} The perturbation friendly test ideals satisfy the \emph{subadditivity property}, i.e. we have
\[
\tau(\KO_\SX,B_1+B_2)\subset \tau(\KO_\SX,B_1)\cdot \tau(\KO_\SX,B_2)
\]
for effective $\Q$-divisors $B_1,B_2$ on $\SX$.

(iv) \emph{(Effective global generation)}
Let $B$ be effective and let $D$ be a divisor such that the $\Q$-divisor $D-K_\SX-B$ is big and nef, and let $H$ be a  globally generated ample divisor. 
Then $\tau(\KO_\SX,B)\otimes \KO_\SX(nH+D)$ is globally generated by $\mathbf B^0(\SX,B, \KO_\SX(nH+D))$ for all $n\geq \mathrm{dim}(\SX\otimes_Rk)$.

(v) \emph{($\tau(\KO_\SX,0)$ is vertical)}
The ideal sheaf $\tau(\KO_\SX,0)$ is vertical, i.e. its support is contained in the special fiber. 
\end{theo}

\begin{proof}
Everything is shown in \cite{schwede-etal-2024}.
We give precise references.
Property (i) holds by \cite[Corollary 7.19(b)]{schwede-etal-2024} and 
(ii) follows from \eqref{plus-formula}, \eqref{define-pert-test-module}, and \eqref{define-pert-test-ideal}.
Statement (iii) is the subadditivity property \cite[Theorem 7.20(e)]{schwede-etal-2024} and (iv) is the effective global generation property \cite[Corollary 7.22]{schwede-etal-2024}.
Recall for (v) that a coherent ideal of $\SX$ is called \emph{vertical} if its support is contained in the special fiber. 
Let $X\coloneqq \SX\otimes_RK$ denote the generic fibre of $\SX$, put $K_X\coloneqq K_\SX|_X$, and recall that the Grauert--Riemenschneider sheaf $\mathscr{J}(X,\omega_X)$ on the regular scheme $X$ equals $\omega_X$ \cite[Definition A.1]{schwede-etal-2024}.
We conclude from \cite[Proposition 7.14(c)]{schwede-etal-2024}, and formula \eqref{noplus-formula} above that
\begin{align*}
\tau(\KO_\SX,0)|_X&=\tau(\omega_\SX,0)\otimes_{\KO_\SX}\KO_\SX(-K_\SX)|_X
=\mathscr{J}(X,\omega_X)\otimes_{\KO_X}\KO_X(-K_X)=\KO_X.
\end{align*}
It follows that the ideal sheaf $\tau(\KO_\SX,0)$ is vertical.
\end{proof}

\begin{defi}\label{defi-pert-friendly}
Let $D$ be a divisor on $\SX$ with linear series $|D|\neq 0$ and $\lambda\in\Q_{>0}$. 
Define the 
\emph{perturbation friendly test ideal of the linear series $|D|$} to be 
\begin{equation}\label{def-pf-test-ideal}
\tau(\KO_\SX,\lambda\cdot |D|):=\sum_{E\in |D|}\tau(\KO_\SX,\lambda\cdot E).
\end{equation}
Thanks to the Noetherian assumption, one can pick a finite number of elements $D_1,\dots, D_r$ in $|D|$ such that $\sum_{i=1}^{r}\tau(\KO_\SX,\lambda\cdot D_i)$ agrees with \eqref{def-pf-test-ideal}.
\end{defi}

\begin{lemma}\label{relate test and base}
If $D$  is a divisor on $\SX$ with linear system $|D|\neq \emptyset$, then 
\begin{equation}\label{relate test and base ideal}
\tau(\KO_\SX, |D|) = \mathfrak{b}_{|D|}\cdot\tau(\KO_\SX, 0)
\end{equation}
where $\mathfrak{b}_{|D|}\coloneqq \mathrm{Im}(H^0(\SX,\KO(D))\otimes_R\KO(-D)\to \KO_\SX)$ denotes the base ideal of the linear system $|D|$.
\end{lemma}

\begin{proof}
Our proof follows Hacon, Lamarche and Schwede \cite[Lemma 7.5(b)]{hls2021}.
Using \eqref{noplus-formula2} and \eqref{def-pf-test-ideal}, one gets
\[
\tau(\KO_\SX, |D|)=\sum_{E\in |D|} \tau(\KO_\SX, E)=\sum_{E\in |D|}\tau(\KO_\SX, 0)\otimes \KO_\SX(-E)=\mathfrak{b}_{|D|}\cdot \tau(\KO_\SX,0),
\]
which shows \eqref{relate test and base ideal}. 
\end{proof}

\begin{defi}
Let $D$ be a $\mathbb{Q}$-divisor with Iitaka dimension $\kappa(D)\geq0$, so $mD$ is a divisor such that $|mD|\neq \emptyset$ for all large and sufficiently divisible $m\in\mathbb{N}$, and $\lambda\in\Q_{>0}$. 
Define the \emph{asymptotic perturbation friendly test ideal} of $|D|$ by
\begin{equation}\label{def-pf-test-ideal-asymptotic-divisor}
\tau(\KO_\SX,\lambda\cdot ||D||)\coloneqq \bigcup_{m>0}\tau\Bigl(\KO_\SX,\frac{\lambda}{m}\cdot |mD|\Bigr).
\end{equation}
\end{defi}

\begin{lemma}\label{asymptotic-pf-test-ideal-property}
The asymptotic test ideals satisfies the following properties:
\begin{enumerate}
\item 
For large and sufficiently divisible $m\in\N$, one has
\begin{equation}\label{stabilize-asymptotic-test-module}
\tau(\KO_\SX,\lambda\cdot ||D||)=\tau\Bigl(\KO_\SX,\frac{\lambda}{m}\cdot |mD|\Bigr).
\end{equation}
\item 
If $|D|\neq \emptyset$, then $\tau(\KO_\SX,\lambda\cdot |D|)\subset \tau(\KO_\SX,\lambda\cdot ||D||) $.
\item 
If $\lambda<\mu$, $(\lambda,\mu)\in\Q_{>0}^2$, then $ \tau(\KO_\SX,\mu\cdot ||D||)\subset \tau(\KO_\SX,\lambda\cdot ||D||)$. 
\item 
If $k\in \mathbb{N}_{>0}$, then $ \tau(\KO_\SX,\lambda\cdot ||D||) =\tau(\KO_\SX,\frac{\lambda}{k}\cdot ||kD||)$.
\end{enumerate}
\end{lemma}

\begin{proof}
The stabilization property (i) follows from the Noetherian assumption if one observes that $\tau(\KO_\SX,\frac{\lambda}{m}|mD|)\subset \tau(\KO_\SX,\frac{\lambda}{m'}|m'D|)$ for all $m,m'\in \N$ where $m$ divides $m'$.
Now Properties (ii)--(iv) follow as in \cite[Lemma 7.5(a),(c),(d),(f)]{hls2021}.
In fact (ii) holds by definition, (iii) is a consequence of Remark \ref{remarks on test ideals}(ii) together with \eqref{define-pert-test-module}, \eqref{def-pf-test-ideal} and \eqref{def-pf-test-ideal-asymptotic-divisor}, and (iv) follows immediately from (i).
\end{proof}

One can show effective global generation and subadditivity also for the asymptotic perturbation friendly global test ideals.

\begin{theo}\label{test-ideal-global-generation-and}
Let $D$ be  a $\mathbb{Q}$-divisor on $\SX$ with Iitaka dimension $\kappa(D)\geq0$.
	
(i) Let $\SH=\KO_\SX(H)$ be a globally generated ample line bundle, 
$E$ be a divisor, and $\lambda\in\Q_{>0}$. 
Let $n=\dim \SX\otimes_Rk$.
If $E-K_\SX-\lambda D$ is big and nef, then
\begin{equation}\label{global-section-with-asymptotic-test-module}
\tau(\KO_\SX,\lambda\cdot ||D||)\otimes\KO_\SX(nH+E)
\end{equation}
is globally generated by a sub linear series of $H^0(\SX, \KO_\SX(nH+E))$ for all $n\geq {\mathrm{dim}(\SX\otimes_Rk)}$.

(ii) For $q,r\in\N_{>0}$, we have
\[
\tau\bigl(\KO_\SX,qr\cdot \|D\|\bigr)\subset \tau\bigl(\KO_\SX,r\cdot \|D\|\bigr)^q.
\]
\end{theo}

\begin{proof}
(i) follows from \cite[Remark 8.36]{schwede-etal-2024}.
For the convenience of the reader we give a proof following \cite[Lemma 7.5 (e)]{hls2021}.
For sufficiently divisible $m\in \N$ we conclude from
\eqref{stabilize-asymptotic-test-module} that
\begin{align*}
\tau(\KO_\SX,\lambda\cdot ||D||)\otimes\KO_\SX(nH+E)
&=\tau\Bigl(\SO_\SX,\frac{\lambda}{m}\cdot |mD|\Bigr)\otimes\KO_\SX(nH+E)\\
&=\sum_{i=1}^{r}\tau\Bigl(\KO_\SX, \frac{\lambda}{m}\cdot D_i\Bigr)\otimes\KO_\SX(nH+E)
\end{align*}
where we pick a finite number of elements $D_1,\dots, D_r$ in $|mD|$ as in Definition \ref{defi-pert-friendly}. 
From Remark \ref{more properties new test modules}(iv) we conclude that
\[
\tau\Bigl(\KO_\SX, \frac{\lambda}{m}\cdot D_i\Bigr)\otimes\KO_\SX(nH+E)
\]
is globally generated by 
$\mathbf B^0(\SX,\frac{\lambda}{m}\cdot D_i, \KO_\SX(nH+E))\subset H^0(\SX,\KO_\SX(nH+E))$ 
for all $n\geq \mathrm{dim}(\SX\otimes_Rk)$.
This finishes the proof of (i).

(ii)
By homogeneity of asymptotic test ideals seen in Lemma \ref{asymptotic-pf-test-ideal-property}(iv), we may assume $r=1$. 
For sufficiently divisible $m$, we deduce from \eqref{stabilize-asymptotic-test-module} that  
\[
\tau\bigl(\SO_\SX,q\cdot ||D||\bigr)=\tau\Bigl(\SO_\SX,\frac{q}{m}\cdot |mD|\Bigr).
\]
For any effective divisor $D_m \sim mD$, using subadditivity in Remark \eqref{more properties new test modules}(iii) for the $\Q$-divisor $\frac{1}{m}D_m$, we get
\[
\tau\Bigl(\SO_\SX,\frac{q}{m}\cdot D_m\Bigr) 
= \tau\Bigl(\SO_\SX,q \cdot \frac{1}{m} D_m\Bigr)
\subset \tau\Bigl(\SO_\SX, \frac{1}{m} D_m\Bigr)^q
\]
and hence
\[
\tau\bigl(\SO_\SX,q\cdot ||D||\bigr)
\subset\tau\Bigl(\SO_\SX,\frac{1}{m}\cdot |mD|\Bigr)^q
\subset\tau\bigl(\KO_\SX,\|D\|\bigr)^q
\]
proving the claim. 
\end{proof}

\section{Continuity of the envelope of the zero function}

Let $K$ be a complete discretely valued field of mixed characteristic $(0,p)$.
Let $X$ be a $n$-dimensional smooth projective variety over $K$ and $L$ an ample line bundle on $X$. 
In the approach of Boucksom--Favre--Jonsson \cite{bfj-singular,bfj-solution} to pluripotential theory on $X^\an$, 
a model $\SL$ of $L$ on the model $\SX$ induces a curvature form $\theta$ on $X^\an$ and we denote by $P_\theta(u)$ the $\theta$-psh envelope of a continuous real function $u$ on $X^\an$, see \cite[2.4--2.6]{gubler-jell-kuennemann-martin} for details.
We follow the strategy from \cite{bfj-singular} and \cite{gubler-jell-kuennemann-martin} to show continuity of the envelope of the zero function $P_\theta(0)$.

\begin{theo} \label{relative Theorem 8.5}
We assume that $(X,L)$ has a model $(\SX, \mathscr A)$ with $\SX$ regular and with $\mathscr A$ an ample line bundle on $\SX$. 
Let $\theta$ be the curvature form on $X^\an$ induced by any model $\SL$ on $\SX$ of $L$ and let $\Fa_{m}$ be the base ideal of $\SL^m$ on $\SX$.  
Then $({m}^{-1}\log|\mathfrak{a}_m|)_{m\in\N_{>0}}$ is a sequence of $\theta$-psh model functions which converges uniformly on $X^\an$ to the envelope $P_\theta(0)$ of the  zero function. 
It follows in particular that $P_\theta(0)$ is continuous.
\end{theo}

\begin{proof}
Consider the graded sequence $(\Fa_{m})_{m>0}$ of base ideals 
\[
\Fa_m=H^0(\SX,\SL^{m})\otimes \SL^{-m}\subset \KO_\SX
\] 
associated with $\SL$. 
Write $\SL=\KO_\SX(D)$ for some divisor $D$ on $\SX$. Let 
\[
\Fb_{m}\coloneqq\tau(\KO_{\SX},m\cdot \|D\|)\subset \KO_\SX 
\] 
denote the associated asymptotic perturbation friendly test ideal of exponent $m$ in mixed characteristic. 
Note that we denote base ideals by $\mathfrak{a}$ and reserve $\mathfrak{b}$ for test ideals, to keep the same notations as in \cite{bfj-singular} and \cite{gubler-jell-kuennemann-martin}. 
Motivated by Lemma \ref{relate test and base ideal} and Remark \ref{more properties new test modules}(v), we consider also the coherent ideals
\[
\Fa_m'\coloneqq\Fa_m\cdot\tau(\KO_\SX,0)\subset\Fa_m \subset \KO_\SX.
\]
These ideals have the following properties:
\begin{enumerate}
\item[(a)]
We have $\Fa_m' \subset \Fb_m$ 
and these coherent ideal sheaves  are vertical for $m$ sufficiently large and divisible.
\item[(b)]
We have $\Fb_{ml} \subset \Fb_m^l$ for all $l,m\in\N_{>0}$.
\item[(c)]
There exists $m_0 \geq 0$ such that $\mathfrak{b}_m\otimes\SA^{\otimes m_0}\otimes\SL^{\otimes m} $ is globally generated for all {$m>0$.}
\end{enumerate}
Indeed these properties can be seen as follows:
by definition of the asymptotic test ideals, we have $\tau(\KO_\SX, |mD|) \subset \Fb_m$. 
For sufficiently large and divisible $m$ we have 
\[
\Fa_m'=\Fa_m\cdot\tau(\KO_\SX,0) = \tau(\KO_\SX, |mD|) \subset \Fb_m
\]
by Lemma \ref{relate test and base}, and the base ideal $\Fa_m$ is vertical as $L$ is ample.
Since $\Fa_m'=\Fa_m\cdot\tau(\KO_\SX,0)$, Remark \ref{more properties new test modules}(v) yields that $\Fa_m'$ is vertical. This proves (a).

Property (b) follows from Theorem \ref{test-ideal-global-generation-and}(ii).
Property (c) is shown as follows.
Choose a divisor $H$ on $\SX$ such that $\KO_\SX(H)$ is ample and globally generated.
We have $n=\mathrm{dim}\,\SX\otimes_Rk$.
As $\SA$ is ample, we can choose $m_0\in \N$ such that the line bundle
\begin{equation}\label{global-gen-eq-1}
\SA^{\otimes m_0}\otimes\KO_\SX\bigl(-K_\SX-(n+1)H\bigr)
\end{equation}
is globally generated.
Given $m\in \N$ we apply Theorem \ref{test-ideal-global-generation-and}(i) to $E\coloneqq mD+H+K_{\SX}$.
Since $E-mD-K_{\SX}=H$ is big and nef, we get that
\begin{equation}\label{global-gen-eq-2}
\mathfrak{b}_m\otimes\KO_\SX(mD+H+K_\SX+n H) 
\end{equation}
is globally generated.
Taking the tensor product of \eqref{global-gen-eq-1} and \eqref{global-gen-eq-2}
we see that $\mathfrak{b}_m\otimes\SA^{\otimes m_0}\otimes\SL^{\otimes m}$ is globally generated.
	
Finally (a), (b) and (c) imply our claim by the strategy of the proof of \cite[Thm.~8.5]{bfj-singular}. 
For convenience of the reader, we give here some details. 
For $m \gg 0$, let 
\[
\varphi_m \coloneqq \frac{1}{m} \log |\Fa_m|
\]
which is a  super-additive sequence of $\theta$-psh functions as shown at the beginning of the proof of \emph{loc.~cit.}. 
Then Step 1 of the quoted proof shows that $P_\theta(0)= \sup_m \varphi_m =\lim_m \varphi_m$ on the quasi-monomial points of $\Xan$. 
Using \cite[Proposition 2.10]{gubler-jell-kuennemann-martin}, this holds pointwise on the whole $\Xan$.
	
We also consider the functions 
\[
\varphi_m' \coloneqq \frac{1}{m}\log |\Fa_m'|= \frac{1}{m}\log|\Fa_m\cdot\tau(\KO_\SX,0)|=\varphi_m +\frac{1}{m}\log|\tau(\KO_\SX,0)|.
\]
It follows from the above that the sequence $\varphi_m'$ also converges pointwise to $P_\theta(0)$ on $\Xan$.

Then Step 2 of \emph{loc.~cit.}~works again in our setting using (a) -- (c) above as follows. 
For $m$ sufficiently large as above and for all $l \in \N_{>0}$, we have seen in (a) and (b) that $\Fa_{ml}' \subset \Fb_{ml} \subset \Fb_m^l$ for all $l \in \N_{>0}$ and hence
\[
\frac{1}{m} \log |\Fb_m| \geq \sup_l \frac{1}{ml} \log|\Fa_{ml}'|=\sup_{l} \varphi_{ml}'
\]
on $\Xan$. 
We conclude that 
\begin{equation} \label{limit inequality}
\frac{1}{m} \log |\Fb_m| \geq \lim_m \varphi_m' =P_\theta(0)
\end{equation}
on $\Xan$. 
The remaining part of the proof of Step 2 is literally the same as in \emph{loc.~cit.} and even simpler as we have Step 1 on $\Xan$ and not only on the quasi-monomial points of $\Xan$. 
Note that we use the global generation property (c) there.
\end{proof}

\section{Resolution of singularities}\label{new-section-6}

Let $K^\circ$ be a complete discrete valuation ring with field of fractions $K$ and $S\coloneqq \Spec\,K^\circ$.
Let $X$ be a smooth projective variety over $K$ of dimension $n$.

\begin{defi} \label{resolution of singularities}
We say that \emph{resolution of singularities holds for projective models of $X$} 
if for every projective model 
$\SX$ of $X$ 
there exists a regular $S$-scheme $\SX'$ and a projective $S$-morphism $\SX'\to \SX$ which induces an isomorphism on $X$.	
\end{defi}

\begin{rem}\label{projective-resolution-gives-good-models}
If one chooses an immersion $X\to \P_K^m$, then the scheme theoretic image $\SX$ of $X$ in $\P^m_S$ defines a projective model of $X$ over $S$. 
If resolution of singularities holds for projective models of $X$, 
then there is a regular projective model $\SX'$ of $X$ over $S$.

 Resolution of singularities holds if $n=1$ \cite[Theorem (1.1)]{artin1986}.
Cossart and Piltant have shown that resolution of singularities holds 
for $n=2$ up to the projectivity of the morphism $\SX'\to \SX$ \cite[Theorem 1.1]{cossart-piltant2019}.
It is only shown in \emph{loc.~cit.}~that this morphism is locally projective.
\end{rem}

It is essential to show that projective models are dominated by SNC-models. 
In order to prove this we are going to use the following assumption.

\begin{defi} \label{embedded resolution of singularities}
We say that {\it embedded  resolution of singularities holds for regular projective models of $X$} 
if for every  regular projective model 
and every proper closed subset $Z$ of $\SX$, there is a projective morphism of $S$-schemes $\pi\colon\SX' \to \SX$ such that the set $\pi^{-1}(Z)$ is the support of a normal crossing divisor and such that $\pi$ is an isomorphism over $\SX \setminus Z$. 
\end{defi}

\begin{lemma} \label{higher geometric model}
Let $\SX$ be  projective model of $X$. 
We assume that resolution of singularities holds for projective models of $X$.
Then for any ample line bundle $L$ on $X$, there exists $m\in \N_{>0}$ and an ample extension $\SL'$ of $L^{\otimes m}$ to a regular $S$-model $\SX'$ of $X$ with a projective morphism $\SX' \to \SX$ over $S$ extending the identity on $X$.
\end{lemma}

\begin{proof}
The arguments are the same as for \cite[Lemma 7.5]{gubler-jell-kuennemann-martin}. 
In a first step, we start with a projective model $\SY$ of $X$ such that $L$ extends to an ample line bundle $\SH$ on $\SY$, possibly replacing $L$ by a positive tensor power. 
By a result of L\"utkebohmert \cite[Lemma 2.2]{luetkebohmert1993}, there is a blowing up $\pi\colon\mathscr Z \to \SY$ in an ideal sheaf  supported in the special fiber of $\SY$  such that the identity on $X$ extends to a morphism $\mathscr Z \to \SX$. 
A property of blowing ups \cite[Proposition II.7.13]{hartshorne-book} shows that $\pi^*(\SH^{\otimes \ell}) \otimes \SO_{\mathscr Z/\SY}(1)$ is an ample line bundle on $\mathscr Z$ with generic fiber $L^{\otimes \ell}$ for sufficiently large $\ell$. 
Replacing $\SX$ by  $\mathscr Z$  and $L$ by $L^{\otimes \ell}$, we conclude that  we may assume that $L$ extends to an ample line bundle $\SH$ on $\SX$. 
	
By a result of P\'epin \cite[Thm.~3.1]{pepin2013}, there is a a blowing-up morphism $\pi' \colon \mathscr Z' \to \SX$ centered in the special fiber of $\SX$ such that $\mathscr Z'$ is semi-factorial. 
So similarly as in the first step, replacing $\SX$ by $\mathscr Z'$ and $L$ by a positive tensor power, we may assume that $L$ extends to an ample line bundle $\SH$ on a semi-factorial projective model $\SX$. 
This is the conclusion of the second step.

Then we apply resolution of singularities to $\SX$ to get a  projective $S$-morphism $\pi \colon \SX' \to \SX$ 
which is an isomorphism on generic fibers and so we may identify the generic fiber of $\SX'$ with $X$. 
Since $\pi$ is projective, there is an $\ell>0$ such that $\SL'\coloneqq \pi^*(\SH^{\otimes \ell}) \otimes \SO_{\SX'/\SX}(1)$ is an ample line bundle on $\SX'$. 
However, we do not know if $\pi$ is a blow up in an ideal supported in the special fiber and hence the restriction $F$ of $ \SO_{\SX'/\SX}(1)$  to the generic fiber $X$ of $\SX'$ might be non-trivial. 
Since $\SX$ is semi-factorial, $F$ extends to a line bundle $\SF$ on $\SX$ and we can replace $ \SO_{\SX'/\SX}(1)$ by $ \SO_{\SX'/\SX}(1) \otimes \varphi^*(\SF^{-1})$. 
Then $\SL'$ is a model of $L^{\otimes \ell}$ proving the claim. 
\end{proof}

\begin{rem} \label{embedded reso of sing gives ample extension to SNC}
If we assume additionally that embedded resolution of singularities holds for regular projective models of $X$, then we can choose $\SX'$ as an SNC-model in Lemma \ref{higher geometric model}.  Indeed, we replace $\SX'$ in the above proof by applying embedded resolution of singularities to the closed subset $\SX_s'$ of the regular projective model $\SX'$ to get an SNC-model of $X$.
\end{rem}

\section{Continuity of the envelope}

Let $K$ be a complete discretely valued field of mixed characteristic $(0,p)$ and $S\coloneqq \Spec ~K^\circ$.
Let $L$ be an ample line bundle on a regular projective variety $X$ of dimension $n$ over $K$.
Let $\theta$ be a closed $(1,1)$-form on $X$ with ample de Rham
class $\{\theta\}$ induced by a model $\SL$ of $L$ on a model $\SX$ of $X$.

\begin{theo}\label{continuity-of-the-envelope}
Assume that resolution of singularities holds   
for projective models of $X$.
If $u\in C^0(X^\an)$, then $P_\theta(u)$ is a uniform limit of $\theta$-psh model functions and thus $P_\theta(u)$ is continuous on $X^\an$.
\end{theo}

\begin{proof}
Using that $u$ is a uniform limit of model functions on $X$, we may assume
that $u$ is itself a model function by \cite[Proposition 2.9(v)]{gubler-jell-kuennemann-martin}.
The model function $u$ is defined by a vertical $\Q$-divisor on a proper model $\SX$ of $X$ over $S$.
By \cite[Proposition 2.9(7)]{gubler-jell-kuennemann-martin},
we may replace $(\theta,u)$ for the proof by $(m\theta,m u)$ for some $m\in \N$. 
Hence we may assume without loss of generality that $u$ is actually defined by a vertical divisor on $\SX$.
If we apply \cite[Lemma 2.2]{luetkebohmert1993} to $\SX$ and a projective model of $X$ as in Remark \ref{projective-resolution-gives-good-models}, then $\SX$ is dominated by a projective model of $X$.
Hence we can assume without loss of generality that $\SX$ is a projective model.
Then we choose a resolution of singularities $\SX'\to \SX$ as in Lemma \ref{higher geometric model} such that some power $L^{\otimes m}$ of  $L$ extends to an ample line bundle $\SA$ on the regular projective model $\SX'$ of $X$.
As before we may assume that $m=1$.
By \cite[Proposition 2.9(4)]{gubler-jell-kuennemann-martin}, we get
\[
P_\theta(u)=P_{\theta+dd^cu}(0)+u.
\]
By construction the class $\theta+dd^cu$ is induced by a line bundle $\SL$ on $\SX'$
whose restriction to $X$ is isomorphic to $L$ and Theorem \ref{relative Theorem 8.5} shows that $P_{\theta+dd^cu}(0)$ is a uniform limit of $(\theta+dd^cu)$-psh model functions.
This finishes our proof.
\end{proof}

\section{The Monge--Amp\`ere equation} \label{section:MA-equation}

Let $K$ be a complete discretely valued field with valuation ring $R$ of mixed characteristic $(0,p)$. 
In this section, we consider a projective regular 
variety $X$ over $K$. 
We assume that resolution of singularities holds for projective models of $X$ and embedded resolution of singularities holds for regular projective models of $X$.
It follows that every projective model of $X$ is dominated by a projective SNC-model. 
By Lemma \ref{higher geometric model} and Remark \ref{embedded reso of sing gives ample extension to SNC}, any ample line bundle on $X$ extends to an ample line bundle on a suitable dominating SNC model after possibly passing to a positive tensor power. 
As explained in \cite[Section 9]{gubler-jell-kuennemann-martin}, these assumptions are enough to set up a pluripotential theory for $\theta$-psh functions with respect to a closed $(1,1)$-form $\theta$ on $\Xan$ with ample de Rham class $\{\theta\}$. All the results of \cite[Sections 1--7]{bfj-singular} hold. 
If we assume continuity of the envelope for ample line bundles on $X$, then monotone regularization \cite[Theorem 8.7]{bfj-singular} holds as well in our setting which is crucial to extend the Monge--Amp\`ere measure $\MA_\theta(\varphi)$ from $\theta$-psh model functions to bounded $\theta$-psh functions on $\Xan$. Then the results from \cite[Sections 4--6]{bfj-singular} hold in our setting by the same arguments. 
 
\begin{theo} \label{Calabi-Yau theorem}
Let $X$ be a smooth projective variety over $K$ 
of dimension $n$ 
and let $\theta$ be a closed $(1,1)$-form on $\Xan$ with ample de Rham class $\{\theta\}$ such that resolution of singularities holds for projective models of $X$ and embedded resolution of singularities holds for regular projective models of $X$.
We consider a positive Radon measure   $\mu$ on $X^\an$ of total mass $\{\theta\}^n$ which is supported on the skeleton of a {projective} SNC-model of $X$. 
Then there is a continuous $\theta$-psh function $\varphi$ on $\Xan$ such that $\MA_\theta(\varphi)=\mu$ and $\varphi$ is unique up to additive constants.
\end{theo}

\begin{proof}
Uniqueness follows from a result of Yuan and Zhang, see \cite[\S 8.1]{bfj-solution}. 
To prove existence of a $\theta$-psh solution $\varphi$, the variational method of Boucksom, Favre, and Jonsson is used. 
Continuity of the envelope for ample line bundles on $X$ holds by Theorem \ref{continuity-of-the-envelope}.
By \cite[Theorems 6.3.2, 6.3.3]{BGJKM}, we conclude that $\theta$ satisfies the crucial \emph{orthogonality property} (see \cite[Definition 7.1]{bfj-solution} for the definition). 
Then existence of a continuous solution follows from \cite[Theorem 8.1]{bfj-solution}.
\end{proof}

\begin{rem} \label{dictionary}
If $L$ is an ample line bundle on $X$ and if $\theta$ is induced by a model $(\SX,\SL)$ of $(X,L)$, then there is a bijective correspondence \cite[\S 2.6]{bfj-solution}
\begin{equation}
\{\theta\mbox{-psh functions}\}\longleftrightarrow\{\mbox{semipositive metrics of } L\},\,\,\,
\varphi\longmapsto \metr_\SL \cdot e^{-\varphi}.
\end{equation}
It follows that Theorem \ref{relative Theorem 8.5} and Theorem \ref{continuity-of-the-envelope} imply Theorem \ref{intro: continuity of the envelope} and that Theorem \ref{Calabi-Yau theorem} implies Theorem \ref{into Calabi-Yau theorem}.
\end{rem}

\addtocontents{toc}{\protect\setcounter{tocdepth}{0}}
\section*{{Acknowledgements}}
\addtocontents{toc}{\protect\setcounter{tocdepth}{1}}
We are grateful to Bhargav Bhatt, Huayi Chen and Mattias Jonsson for remarks about previous versions of this paper.
We are especially grateful to Karl Schwede for discussions and help about subadditivity of global $+$-test ideals and  perturbation friendly global test ideals.
We are grateful to the referee for helpful remarks and suggestions.

\bibliographystyle{alpha}


\newcommand{\etalchar}[1]{$^{#1}$}

\end{document}